\documentclass[a4paper, 12pt, reqno]{amsart}

\pdfoutput=1

\usepackage[utf8]{inputenc}
\usepackage[margin=2.5cm, headsep=0.75cm, footskip=0.75cm]{geometry}
\usepackage{amsmath, amsthm, amssymb, mathtools, physics}
\usepackage[hang, flushmargin]{footmisc}
\usepackage[english]{isodate}
\usepackage{enumitem}
\usepackage{xcolor}
\usepackage[colorlinks=true, linkcolor=blue, linkbordercolor=blue, citecolor=red, citebordercolor=red, urlcolor=indigo, linktocpage=true]{hyperref}
\usepackage[capitalise, nameinlink, noabbrev, nosort]{cleveref}
\usepackage{doi}
\usepackage{tikz, tikzsymbols, tikz-cd}
\usepackage[normalem]{ulem}

\definecolor{indigo}{HTML}{492DA5}

\allowdisplaybreaks

\setenumerate{listparindent=\parindent}

\providecommand{\noopsort}[1]{}

\makeatletter\g@addto@macro\bfseries{\boldmath}\makeatother

\makeatletter
\let\origsection\section
\renewcommand\section{\@ifstar{\starsection}{\nostarsection}}
\newcommand\sectionspace{\vspace{0.5ex}}
\newcommand\nostarsection[1]{\sectionspace\origsection{#1}\sectionspace}
\newcommand\starsection[1]{\sectionspace\origsection*{#1}\sectionspace}
\makeatother

\setlist[enumerate]{font=\normalfont}
\crefname{enumi}{}{}
\crefname{enumii}{}{}

\numberwithin{equation}{section}
\crefname{equation}{equation}{equations}

\newtheorem{theorem}{Theorem}[section]

\newtheorem{thm}[theorem]{Theorem}
\crefname{thm}{Theorem}{Theorems}

\newtheorem{lemma}[theorem]{Lemma}
\crefname{lemma}{Lemma}{Lemmas}

\newtheorem{prop}[theorem]{Proposition}
\crefname{prop}{Proposition}{Propositions}

\newtheorem{cor}[theorem]{Corollary}
\crefname{cor}{Corollary}{Corollaries}

\theoremstyle{definition}

\newtheorem{definition}[theorem]{Definition}
\crefname{definition}{Definition}{Definitions}

\theoremstyle{remark}

\newtheorem{example}[theorem]{Example}
\crefname{example}{Example}{Examples}

\newcommand{\hl}[1]{\textcolor{magenta}{\emph{#1}}}
\newcommand{\C}{\mathbb{C}}
\newcommand{\N}{\mathbb{N}}
\newcommand{\T}{\mathbb{T}}
\newcommand{\Z}{\mathbb{Z}}
\newcommand{\FF}{\mathcal{F}}
\newcommand{\Go}{G^{(0)}}
\newcommand{\Gc}{G^{(2)}}
\newcommand{\Sigmao}{\Sigma^{(0)}}
\newcommand{\vecspan}{\operatorname{span}}
\newcommand{\supp}{\operatorname{supp}}
\newcommand{\osupp}{\supp^{\circ}}
\newcommand{\id}{\operatorname{id}}
\newcommand{\Iso}{\operatorname{Iso}}
\renewcommand{\th}{\mathrm{th}}
\newcommand{\restr}[1]{\ensuremath{\vert_{#1}}}
\newcommand{\rnorm}[1]{\ensuremath{\norm{#1}_r}}

\hypersetup{pdfauthor={Becky Armstrong, Jonathan H. Brown, Lisa Orloff Clark, Kristin Courtney, Ying-Fen Lin, Kathryn McCormick, and Jacqui Ramagge}, pdftitle={The local bisection hypothesis for twisted groupoid C*-algebras}}

\begin{document}

\title[The local bisection hypothesis for twisted groupoid C*-algebras]{The local bisection hypothesis for twisted groupoid C*-algebras}

\author[Armstrong]{Becky Armstrong}
\author[Brown]{Jonathan H.\ Brown}
\author[Clark]{Lisa Orloff Clark}
\author[Courtney]{Kristin Courtney}
\author[Lin]{Ying-Fen Lin}
\author[McCormick]{Kathryn McCormick}
\author[Ramagge]{Jacqui Ramagge}

\date{\today}

\address[B.\ Armstrong and L.O.\ Clark]{School of Mathematics and Statistics, Victoria University of Wellington, PO Box 600, Wellington 6140, NEW ZEALAND}
\email[B.\ Armstrong]{\href{mailto:becky.armstrong@vuw.ac.nz}{becky.armstrong@vuw.ac.nz}}
\email[L.O.\ Clark]{\href{mailto:lisa.orloffclark@vuw.ac.nz}{lisa.orloffclark@vuw.ac.nz}}
\address[J.H.\ Brown]{Department of Mathematics, University of Dayton, 300 College Park, Dayton OH, 45469, UNITED STATES}
\email{\href{mailto:jonathan.henry.brown@gmail.com}{jonathan.henry.brown@gmail.com}}
\address[K.\ Courtney]{Mathematical Institute, University of M\"unster, Einsteinstr.\ 62, 48149 M\"unster, GERMANY}
\email{\href{mailto:kcourtne@uni-muenster.de}{kcourtne@uni-muenster.de}}
\address[Y.-F.\ Lin]{Mathematical Sciences Research Centre, Queen's University Belfast, Bel\-fast, BT7 1NN, UNITED KINGDOM}
\email{\href{mailto:y.lin@qub.ac.uk}{y.lin@qub.ac.uk}}
\address[K.\ McCormick]{Department of Mathematics and Statistics, California State University, Long Beach, 1250 Bellflower Blvd, Long Beach CA, 90840, UNITED STATES}
\email{\href{mailto:kathryn.mccormick@csulb.edu}{kathryn.mccormick@csulb.edu}}
\address[J.\ Ramagge]{University of South Australia, Adelaide, AUSTRALIA}
\email{\href{mailto:jacqui.ramagge@unisa.edu.au}{jacqui.ramagge@unisa.edu.au}}


\thanks{This research began at the ICMS in Edinburgh and was supported by their Research-in-Groups funding. The research was additionally funded by the Marsden Fund of the Royal Society of New Zealand (grant number 21-VUW-156); California State University, Long Beach; the Deutsche Forschungsgemeinschaft (DFG, German Research Foundation) under Germany's Excellence Strategy -- EXC 2044 -- 390685587, Mathematics M\"unster -- Dynamics -- Geometry -- Structure; the Deutsche Forschungsgemeinschaft (DFG, German Research Foundation) -- Project-ID 427320536 -- SFB 1442; and ERC Advanced Grant 834267 -- AMAREC. We also extend our thanks to Victoria University of Wellington for their hospitality during various author visits to the third-named author}

\begin{abstract} In this note, we present criteria that are equivalent to a locally compact Hausdorff groupoid $G$ being effective. One of these conditions is that $G$ satisfies the \emph{C*\nobreakdash-algebraic local bisection hypothesis}; that is, that every normaliser in the reduced twisted groupoid C*-algebra is supported on an open bisection. The semigroup of normalisers plays a fundamental role in our proof, as does the semigroup of normalisers in cyclic group C*-algebras.
\end{abstract}

\maketitle

\section{Introduction}

The connection between twisted groupoid C*-algebras and diagonal and Cartan pairs has become a cornerstone of C*-algebraic theory. Given a C*-algebra $A$ and a commutative subalgebra $B$, Kumjian and Renault show in \cite{Kumjian1986} and \cite{Renault2008}, respectively, that:
\begin{itemize}
\item $(A,B)$ is a \emph{diagonal} pair if and only if there is a twisted groupoid $\Sigma \to G$ with $G$ \emph{principal} such that $A \cong C_r^*(G;\Sigma)$ and $B \cong C_0(\Go)$; and
\item $(A,B)$ is a \emph{Cartan} pair if and only if there is a twisted groupoid $\Sigma \to G$ with $G$ \emph{effective} such that $A \cong C_r^*(G;\Sigma)$ and $B \cong C_0(\Go)$.
\end{itemize}

Twisted Steinberg algebras were introduced in \cite{ACCLMR2022} as a purely algebraic analogue to twisted groupoid C*-algebras, and algebraic versions of Kumjian and Renault's findings were established in \cite{ACCCLMRSS2023}. The authors (which include most of the authors of the current note) define \emph{algebraic diagonal} and \emph{algebraic Cartan pairs} $(A_0,B_0)$ in \cite{ACCCLMRSS2023}, and they prove that $A_0$ is isomorphic to the twisted Steinberg algebra $A(G;\Sigma)$ of a ``discrete'' twist $\Sigma$ over an effective ample Hausdorff groupoid $G$ (which is principal if $B_0$ is an algebraic diagonal subalgebra of $A_0)$. An algebraic Cartan subalgebra $B_0$ of $A_0$ is a maximal commutative subalgebra (just like a C*-algebraic Cartan subalgebra).

The authors of \cite{ACCCLMRSS2023} also define \emph{algebraic quasi-Cartan pairs}, in which the maximality requirement is relaxed to requiring that the conditional expectation $E\colon A_0 \to B_0$ be \hl{implemented by idempotents}, in the sense that for each algebraic normaliser $n \in A_0$, there exists an idempotent $p \in B_0$ such that $E(n) = pn = np$. This corresponds to a relaxation of the effectiveness assumption as follows: a discrete twist $(\Sigma,i,q)$ over an ample Hausdorff groupoid $G$ satisfies the \hl{local bisection hypothesis} \cite[Definition~4.1]{ACCCLMRSS2023} if, for every normaliser $n \in A(G;\Sigma)$ of $A(\Go)$, $q(\osupp(n))$ is an open bisection of $G$. Several examples are provided to demonstrate that there are indeed discrete twists over ample Hausdorff groupoids that satisfy the local bisection hypothesis but for which the groupoid is not effective (see \cite[Section~9]{ACCCLMRSS2023}). However, every discrete twist over an effective ample groupoid does satisfy the local bisection hypothesis (see \cite[Lemma~4.7(b)]{ACCCLMRSS2023}). The authors then show that:
\begin{itemize}
\item $(A_0,B_0)$ is an \emph{algebraic quasi-Cartan} pair if and only if there is a discrete twisted groupoid $\Sigma \to G$ that satisfies the local bisection hypothesis such that $A_0 \cong A(G;\Sigma)$, and $B_0 \cong A(\Go)$.
\end{itemize}
Thus, a natural question arises: what does the C*-algebraic local bisection hypothesis tell us about a twisted groupoid? Is this a more general property than effectiveness? It turns out that the answer is no. We show that, unlike in the algebraic setting, if we insist that each normaliser in the reduced twisted groupoid C*-algebra $C_r^*(G;\Sigma)$ is supported on the preimage under the quotient map of an open bisection of $G$ (when viewed as a function on $\Sigma$ via the map $j$; see, for example, \cite[Proposition~2.8]{BFPR2021}), then $G$ is effective. That is, we prove the converse to Renault's \cite[Proposition~4.8(ii)]{Renault2008}. With this characterisation, we see that being quasi-Cartan but not Cartan is a purely algebraic phenomenon.

\subsection*{Outline}

We begin by discussing the theory of topological groupoids, twisted groupoid C*\nobreakdash-algebras, and Cartan pairs. We then state our main result in \cref{thm: main}, which is a list of conditions that are equivalent to a groupoid being effective. We also present the special case of the groupoid being ample, in which effectiveness is equivalent to the conditional expectation onto the Cartan subalgebra being ``implemented by projections'' (see \cref{cor: ample}). (This is similar to the condition (AQP) in \cite{ACCCLMRSS2023} of the algebraic conditional expectation being implemented by idempotents.) We then show in \cref{sec: example} that the discrete group of integers, which is an ample groupoid, satisfies the algebraic local bisection hypothesis but does not satisfy the C*-algebraic version. We use this, along with an analysis of normalisers in finite cyclic group C*-algebras, to complete the proof of \cref{thm: main} in \cref{sec: Cstar-LBH implies effectiveness} (see \cref{prop: untwisted LBH implies effective}).

\section{Preliminaries}

In this section, we recall some terminology relating to groupoids, twists, twisted groupoid C*-algebras, and Cartan pairs. For groupoids, twists, and twisted groupoid C*-algebras, we primarily follow the conventions and notation of \cite{Sims2020}.

\subsection{Groupoid terminology}

A \hl{groupoid} $G$ is a small category in which every morphism $\gamma \in G$ has a unique inverse $\gamma^{-1} \in G$. We write $\Go$ for the set of units (or objects) in $G$, and we write $\Gc$ for the collection of composable pairs in $G \times G$. We read composition of elements from right to left, so the range and source maps $r,s\colon G \to \Go$ are given by $r(\gamma) = \gamma \gamma^{-1}$ and $s(\gamma) = \gamma^{-1} \gamma$, respectively. For each $x, y \in \Go$, we define
\[
G^x \coloneqq \{ \gamma \in G : r(\gamma) = x \}, \quad G_y \coloneqq \{ \gamma \in G : s(\gamma) = y \}, \quad \text{and } \quad G_y^x \coloneqq G^x \cap G_y.
\]

We call a groupoid $G$ a \hl{topological groupoid} if it is endowed with a topology with respect to which composition and inversion are continuous, and the unit space $\Go$ is Hausdorff. In what follows, we will restrict our attention to locally compact Hausdorff groupoids; however, we do not require our groupoids to be second-countable. We say that a topological groupoid $G$ is \hl{\'etale} if the range map $r\colon G \to \Go$ is a local homeomorphism (or, equivalently, if the source map $s\colon G \to \Go$ is a local homeomorphism). If $G$ is Hausdorff, then $\Go$ is closed, and if $G$ is \'etale, then $\Go$ is open. We call a subset $B$ of $G$ a \hl{bisection} if it is contained in an open subset $U$ of $G$ such that $r\restr{U}$ and $s\restr{U}$ are homeomorphisms onto open subsets of $\Go$. Every locally compact Hausdorff \'etale groupoid has a basis of precompact open bisections. We call a topological groupoid \hl{ample} if it has a basis of \emph{compact} open bisections. Given subsets $A$ and $B$ of a Hausdorff \'etale groupoid $G$, we define
\[
A^{-1} \coloneqq \{ \alpha^{-1} : \alpha \in A \} \quad \text{ and } \quad AB \coloneqq \{ \alpha\beta : (\alpha, \beta) \in (A \times B) \cap \Gc \}.
\]
Note that if $A$ and $B$ are open bisections of $G$, then so are $A^{-1}$ and $AB$, and in this case, $AA^{-1} = r(A)$ and $A^{-1}A = s(A)$.

The \hl{isotropy} of a groupoid $G$ is the subgroupoid
\[
\Iso(G) \coloneqq \bigsqcup_{x \in \Go} G_x^x = \{ \gamma \in G : r(\gamma) = s(\gamma) \}.
\]
We denote the topological interior of the isotropy of $G$ by $\Iso(G)^\circ$, and we observe that if $G$ is a locally compact Hausdorff \'etale groupoid, then so is $\Iso(G)^\circ$. We say that $G$ is \hl{principal} if $\Iso(G) = \Go$, and that $G$ is \hl{effective} if $\Iso(G)^\circ = \Go$. A number of equivalent (and non-equivalent but related) conditions to effectiveness are given in, for instance, \cite[Lemma~3.1]{BCFS2014} and \cite[Remark~2.1]{ACCCLMRSS2023}.

\subsection{Twists}

Let $G$ be a locally compact Hausdorff \'etale groupoid. A \hl{twist} $\Sigma$ over $G$ (often called a \hl{twisted groupoid} and denoted by $(\Sigma,i,q)$ or by $\Sigma \to G$) is a sequence
\[
\Go \times \T \overset{i}{\hookrightarrow} \Sigma \overset{q}{\twoheadrightarrow} G,
\]
where $\Go \times \T$ is a trivial group bundle with fibres $\T$, and $\Sigma$ is a Hausdorff groupoid with unit space $\Sigmao = i(\Go \times \{1\})$, such that
\begin{itemize}
\item $i$ and $q$ are continuous groupoid homomorphisms that restrict to homeomorphisms of unit spaces;
\item the sequence is \hl{exact}, in the sense that $i$ is injective, $q$ is surjective, and $q^{-1}(\Go) = i(\Go \times \T)$;
\item $i(\Go \times \T)$ is \hl{central} in $\Sigma$, in the sense that $i(r(\sigma),z) \,\sigma = \sigma \, i(s(\sigma),z)$ for all $\sigma \in \Sigma$ and $z \in \T$; and
\item $\Sigma$ is a locally trivial $\T$-bundle over $G$, in the sense that for every $\alpha \in G$, there is an open bisection $U_\alpha$ of $G$ such that $\alpha \in U_\alpha$, and there is a continuous map $S_\alpha\colon U_\alpha \to \Sigma$ (called a \hl{continuous local section}) such that $q \circ S_\alpha = \id_{U_\alpha}$ and $U_\alpha \times \T \ni (\beta, z) \mapsto i(r(\beta),z) \, S_\alpha(\beta) \in q^{-1}(U_\alpha)$ is a homeomorphism.
\end{itemize}

Given a twisted groupoid $\Sigma \to G$, we define $z \cdot \sigma \coloneqq i(r(\sigma),z) \,\sigma = \sigma \, i(s(\sigma),z)$ for all $\sigma \in \Sigma$ and $z \in \T$. For $\sigma, \varepsilon \in \Sigma$, we have $q(\sigma) = q(\varepsilon)$ if and only if there is a (necessarily unique) $z \in \T$ such that $\sigma = z \cdot \varepsilon$. For each $x \in \Go$, we have $q^{-1}(x) = i(\{x\} \times \T)$.

There are variations on how twisted groupoids are defined in the literature, but these definitions are generally equivalent; see, for example, \cite[Remark~2.6]{Armstrong2022}. Twists over principal groupoids are of particular importance because they characterise C*-diagonals; see Kumjian's definition \cite{Kumjian1986} and subsequent papers \cite{MW1992, aHKS2011, CaH2012}. Twists over effective groupoids are also of significant interest because they characterise Cartan pairs; see Renault's work \cite[Theorems~5.2 and 5.9]{Renault2008}.

\subsection{Twisted groupoid C*-algebras}

Given a locally compact Hausdorff space $X$ and a continuous function $f \in C(X)$, we define $\osupp(f) \coloneqq f^{-1}(\C {\setminus} \{0\})$ and $\supp(f) \coloneqq \overline{\osupp(f)}$. We say that $f \in C(X)$ is \hl{compactly supported} if $\supp(f)$ is compact, and we write $C_c(X)$ for the collection of compactly supported functions in $C(X)$. We say that $f \in C(X)$ \hl{vanishes at infinity} if, for every $\epsilon > 0$, the set $\{ x \in X : \abs{f(x)} \ge \epsilon \}$ is compact. We write $C_0(X)$ for the collection of functions in $C(X)$ that vanish at infinity, and we note that $C_0(X)$ is the C*-completion of $C_c(X)$ with respect to the uniform norm $\norm{\cdot}_\infty$.

Let $G$ be a locally compact Hausdorff \'etale groupoid. Then $C_c(G)$ is a *-algebra under the operations
\[
(f * g)(\gamma) \coloneqq \sum_{\eta \in G^{r(\gamma)}} f(\eta) g(\eta^{-1}\gamma) \quad \text{ and } \quad f^*(\gamma) \coloneqq \overline{f(\gamma^{-1})}.
\]
For each $x \in \Go$, let $\pi_x\colon C_c(G) \to B(\ell^2(G_x))$ denote the regular representation of $C_c(G)$ at $x$, given by $\pi_x(f)g = f * g$. The completion $C_r^*(G)$ of $C_c(G)$ with respect to the \hl{reduced norm} $\rnorm{f} \coloneqq \sup\{ \norm{ \pi_x(f)} : x \in \Go \}$ is called the \hl{reduced groupoid C*-algebra} of $G$. The Steinberg algebra $A(G)$ (see \cite[Section~4]{Steinberg2010} and \cite[Section~3]{CFST2014}) is dense in both $C_c(G)$ and $C_r^*(G)$ with respect to the reduced norm. For $f \in C_c(G)$ with $\supp(f) \subseteq \Go$, we have $\rnorm{f} = \norm{f}_\infty$, and so we identify the completion of $C_c(\Go) \subseteq C_c(G)$ with respect to the reduced norm with $C_0(\Go)$.

The construction of twisted groupoid C*-algebras is similar. Let $\Sigma$ be a twist over $G$. We say that $f \in C_c(\Sigma)$ is \hl{$\T$-contravariant} if $f(\sigma \cdot z) = \overline{z} \cdot f(\sigma)$ for all $\sigma \in \Sigma$ and $z \in \T$. Define $C_c(G;\Sigma) \coloneqq \{ f \in C_c(\Sigma) : f \text{ is $\T$-contravariant} \}$, and for each $x \in \Go$, define $L^2(G_x;\Sigma_x) \coloneqq \{ f \in L^2(\Sigma_x) : f \text{ is $\T$-contravariant} \}$. For each $x \in \Go$, let $\pi_x\colon C_c(G;\Sigma) \to B(L^2(G_x;\Sigma_x))$ denote the regular representation of $C_c(G;\Sigma)$ at $x$, given by $\pi_x(f)g = f * g$. The completion $C_r^*(G;\Sigma)$ of $C_c(G;\Sigma)$ with respect to the \hl{reduced norm} $\rnorm{f} \coloneqq \sup\{ \norm{ \pi_x(f)} : x \in \Go \}$ is called the \hl{reduced twisted groupoid C*-algebra} of the pair $(G,\Sigma)$. The twisted Steinberg algebra $A(G;\Sigma)$ (see \cite[Section~2.4]{ACCCLMRSS2023}) is dense in both $C_c(G;\Sigma)$ and $C_r^*(G;\Sigma)$ with respect to the reduced norm. If $\Sigma = G \times \T$, then $C_r^*(G;\Sigma) \cong C_r^*(G)$ and $A(G;\Sigma) \cong A(G)$. The groupoid $q^{-1}(\Go)$ is a twist over $\Go$, and the completion of $C_c(\Go;q^{-1}(\Go)) \subseteq C_c(G;\Sigma)$ with respect to the reduced norm is isomorphic to $C_0(\Go)$. We write $C_0(\Go) \subseteq C_r^*(G;\Sigma)$ with this identification in mind. Note that in some literature (such as \cite{BFPR2021}) the reduced C*-algebra of a twisted groupoid $\Sigma \to G$ is denoted by $C_r^*(\Sigma;G)$ and is defined using continuous sections of a complex line bundle instead (see \cite{Renault2008}).

Given a twisted groupoid $\Sigma \to G$, we write $C_0(G;\Sigma)$ for the collection of $\T$-contravariant functions in $C_0(\Sigma)$. As used in \cite{Renault2008} and as detailed in \cite{BFPR2021}, we can view elements of $C_r^*(G;\Sigma)$ as $\T$-contravariant functions in $C_0(G;\Sigma)$ via a norm-decreasing injective linear map $j\colon C_r^*(G;\Sigma) \to C_0(G;\Sigma)$ that preserves convolution and involution and satisfies $j(f) = f$ for all $f \in C_c(G;\Sigma)$. Given $a \in C_r^*(G;\Sigma)$, we define $S_a \coloneqq q(\osupp(j(a)))$. This is consistent with the definition of $\supp'(a)$ in \cite[Section~4]{Renault2008}; see also \cite[Remark~2.4]{BFPR2021}. Note that $S_a$ is open in $G$ because $j(a)$ is continuous and $q$ is an open map (see \cite[Lemma~2.7(a)]{Armstrong2022}). Similarly, in the untwisted setting, we can view elements of $C_r^*(G)$ as functions in $C_0(G)$ via the norm-decreasing injective linear map $j\colon C_r^*(G) \to C_0(G)$ extending the identity map on $C_c(G)$ (see \cite[Proposition~II.4.2]{Renault1980} for details). In this setting, given $a \in C_r^*(G)$, we define $S_a \coloneqq \osupp(j(a))$, which is again open in $G$.

\subsection{Cartan pairs}

Let $A$ be a C*-algebra and let $B$ be an abelian subalgebra of $A$. We call an element $n \in A$ a \hl{normaliser} of $B$ if $nBn^* \cup n^*Bn \subseteq B$.
We write $N(B)$ for the set of normalisers of $B$ in $A$. Note that $n \in A$ is a normaliser of $B$ if and only if $n^*$ is a normaliser of $B$.

There are various equivalent definitions of a \emph{Cartan pair} in the literature; we use the one from \cite[Definition~1.1]{BFPR2021}.

\begin{definition}
Let $A$ be a C*-algebra and let $B$ be an abelian subalgebra of $A$. We call $(A,B)$ a \hl{Cartan pair} and say that $B$ is a \hl{Cartan subalgebra} of $A$ if the following conditions are satisfied:
\begin{enumerate}[label=(\roman*)]
\item $B$ is a maximal abelian subalgebra of $A$;
\item $\vecspan{(N(B))} = A$; and
\item there exists a faithful conditional expectation $E\colon A \to B$.
\end{enumerate}
\end{definition}

Most definitions of a Cartan pair that appear in the literature include the additional assumption that $B$ contains an approximate identity for $A$. However, Pitts shows in \cite[Theorem~2.6]{Pitts2022} that this condition follows automatically from the other three. It follows that for all $n \in N(B)$, we have $nn^*, n^*n \in B$.

C*-algebras containing Cartan subalgebras are precisely the twisted C*-algebras of effective Hausdorff \'etale groupoids (see \cref{thm: Cartan subalgebras} below). This was proven in the separable setting by Renault \cite{Renault2008} and was later extended to the non-separable setting by Raad \cite{Raad2022}.

\begin{thm}[{\cite[Theorems~5.2 and 5.9]{Renault2008} and \cite[Theorem~1.2]{Raad2022}}] \label{thm: Cartan subalgebras}
Let $\Sigma$ be a twist over an effective locally compact Hausdorff \'etale groupoid $G$. Then the map that restricts functions in $C_c(G;\Sigma)$ from $\Sigma$ to $q^{-1}(\Go)$ extends to a faithful conditional expectation $E\colon C_r^*(G;\Sigma) \to C_r^*(\Go; q^{-1}(\Go)) \cong C_0(\Go)$, and $C_0(\Go)$ is a Cartan subalgebra of $C_r^*(G;\Sigma)$. Conversely, if $(A,B)$ is a Cartan pair, then there exists a twist $\Sigma$ over an effective locally compact Hausdorff \'etale groupoid $G$ such that there is an isomorphism from $A$ to $C_r^*(G;\Sigma)$ that maps $B$
onto $C_0(\Go)$.
\end{thm}

\section{Effective groupoids}

In \cite[Proposition~4.8(i)]{Renault2008} Renault shows that if $(\Sigma,i,q)$ is a twist over a second-countable locally compact Hausdorff \'etale groupoid $G$, then for any element $a \in C_r^*(G;\Sigma)$ such that $S_a = q(\osupp(j(a)))$ is an open bisection of $G$, we have that $a$ is a normaliser of $C_0(\Go)$, regardless of whether $C_0(\Go)$ is a Cartan subalgebra of $C_r^*(G;\Sigma)$. Raad \cite{Raad2022} also makes this observation in the non-second-countable setting.

In this paper, we say that a twist $\Sigma$ over a Hausdorff \'etale groupoid $G$ satisfies the \hl{C*\nobreakdash-algebraic local bisection hypothesis} if, for every normaliser $n \in C_r^*(G;\Sigma)$ of $C_0(\Go)$, $S_n$ is an open bisection of $G$. Renault shows in \cite[Proposition~4.8(ii)]{Renault2008} that the C*\nobreakdash-algebraic local bisection hypothesis holds for twists over effective groupoids. The purpose of our paper is to show that the C*-algebraic local bisection hypothesis is in fact \emph{equivalent} to the groupoid being effective.

Throughout this section, let $\iota\colon C_0(\Go) \to C_r^*(\Go;q^{-1}(\Go))$ be the isomorphism sending $f \in C_c(\Go)$ to the function $z \cdot x \mapsto \overline{z} f(x)$, and let $E\colon C_r^*(G;\Sigma) \to C_0(\Go)$ be the conditional expectation extending restriction of functions.

The remainder of the paper is devoted to establishing the following theorem.

\begin{thm} \label{thm: main}
Let $G$ be a locally compact Hausdorff \'etale groupoid. The following are equivalent.
\begin{enumerate}[label=(\arabic*)]
\item \label{it: effective} $G$ is effective.
\item \label{it: Cartan} For any twist $(\Sigma,i,q)$ over $G$, $\big(C_r^*(G;\Sigma), C_0(\Go)\big)$ is a Cartan pair.
\item \label{it: twisted LBH} For any twist $(\Sigma,i,q)$ over $G$ and any normaliser $n \in C_r^*(G;\Sigma)$ of $C_0(\Go)$, $S_n = q(\osupp(j(n)))$ is an open bisection of $G$.
\item \label{it: cond exp} For any twist $(\Sigma,i,q)$ over $G$ and any normaliser $n \in C_r^*(G;\Sigma)$ of $C_0(\Go)$, there exists a sequence $(f_k)_{k=1}^\infty \subseteq C_0(\Go)$ such that for each $k \ge 1$,
\[
\iota(f_k)j(n) = \iota(f_k E(n)) = \iota(E(n) f_k) = j(n) \iota(f_k) \in C_r^*(\Go;q^{-1}(\Go)),
\]
and
\[
\iota(E(n)) = \lim_{k \to \infty} (\iota(f_k) j(n)).
\]
\item \label{it: untwisted LBH} For each normaliser $n \in C_r^*(G)$ of $C_0(\Go)$, $S_n = \osupp(j(n))$ is an open bisection of $G$.
\end{enumerate}
\end{thm}

\begin{proof}
That \cref{it: effective,it: Cartan} are equivalent (without assuming second-countability) follows from \cite[Corollary~7.6]{KM2020}. That \cref{it: effective} implies \cref{it: twisted LBH} follows from \cite[Lemma~5.4]{BFPR2021}, which is a generalisation of Renault's \cite[Proposition~4.8(ii)]{Renault2008}.

\Cref{it: twisted LBH} $\implies$ \cref{it: cond exp}:
Let $(\Sigma,i,q)$ be a twist over $G$, and fix a normaliser $n \in C_r^*(G;\Sigma)$ of $C_0(\Go)$. For each $k \in \N {\setminus} \{0\}$, let
\[
W_k \coloneqq \{ x \in \Go : \abs{E(n)(x)} > \tfrac{1}{k} \} = E(n)^{-1}\big(\C \setminus \overline{B}(0,\tfrac{1}{k})\big),
\]
where $\overline{B}(0,\tfrac{1}{k}) \subseteq \C$ is the closed ball of radius $\tfrac{1}{k}$ centred at $0$. Then each $W_k$ is open since $E(n)$ is continuous. Moreover, since $E(n) \in C_0(\Go)$, the set
\[
V_k \coloneqq \{ x \in \Go : \abs{E(n)(x)} \ge \tfrac{1}{k} \}
\]
is compact and hence closed, and so $\overline{W_k} \subseteq V_k$. Hence each $\overline{W_k}$ is compact and is contained in the set $S_{E(n)} = \osupp(E(n)) = S_n \cap \Go$. For each $k \ge 1$, use Urysohn's lemma to choose $f_k \in C_c(\Go,[0,1])$ such that $\supp(f_k) \subseteq S_{E(n)}$ and $f_k(x) = 1$ for all $x \in \overline{W_k}$. Since $q(\osupp(j(n))) = S_n$ is a bisection of $G$, it follows that
\[
q(\osupp(\iota(f_k) j(n))) \subseteq (S_n \cap \Go) S_n \subseteq S_n \cap \Go = S_{E(n)},
\]
and thus, since $C_0(\Go)$ is abelian, we have
\[
\iota(f_k) j(n) = \iota(f_k E(n)) = \iota(E(n) f_k) = j(n) \iota(f_k) \in C_r^*(\Go;q^{-1}(\Go)).
\]
We now show that the sequence $(E(n) f_k)_{k=1}^\infty$ converges to $E(n)$ in the uniform norm. For each $k \ge 1$, we have
\begin{align*}
\norm{E(n) f_k - E(n)}_\infty &= \sup{\big\{ \abs{(E(n) f_k)(x) - E(n)(x)} : x \in \Go \big\}} \\
&= \sup{\big\{ \abs{E(n)(x)} \, \abs{f_k(x) - 1} : x \in S_{E(n)} {\setminus} \overline{W_k}\big \}}.
\end{align*}
For each $k \ge 1$ and all $x \in S_{E(n)} {\setminus} \overline{W_k}$, we have $x \notin W_k$, and so $\abs{E(n)(x)} \le \frac{1}{k}$ by the definition of $W_k$, and $\abs{f_k(x) - 1} \le 1$ since $f_k \in C_c(\Go,[0,1])$. Therefore, continuing the calculation above, we see that
\[
\norm{E(n) f_k - E(n)}_\infty \le \sup{\big\{ \abs{E(n)(x)} : x \in S_{E(n)} {\setminus} \overline{W_k} \big\}} \le \tfrac{1}{k}
\]
for each $k \ge 1$, and so $\norm{E(n) f_k - E(n)}_\infty \to 0$ as $k \to \infty$.

\Cref{it: cond exp} $\implies$ \cref{it: untwisted LBH}: Suppose for contradiction that there exists a normaliser $n \in C_r^*(G)$ of $C_0(\Go)$ such that $S_n$ is not a bisection of $G$. Then there exist distinct elements $\gamma_1, \gamma_2 \in S_n$ such that $s(\gamma_1) = s(\gamma_2)$ or $r(\gamma_1) = r(\gamma_2)$. Since $n^* \in C_r^*(G)$ is also a normaliser of $C_0(\Go)$, we may assume without loss of generality that $r(\gamma_1) = r(\gamma_2)$, because if $s(\gamma_1) = s(\gamma_2)$, then $r(\gamma_1^{-1}) = r(\gamma_2^{-1})$, and $\gamma_1^{-1}, \gamma_2^{-1} \in S_n^{-1} = S_{n^*}$. Let $B \subseteq G$ be an open bisection of $G$ such that $\gamma_1^{-1} \in B$, and use Urysohn's lemma to choose $g \in C_c(G)$ such that $\supp(g) \subseteq B$ and $g(\gamma_1^{-1}) = 1$. Then $g$ is a normaliser of $C_0(\Go)$, because for all $h \in C_0(\Go)$, we have
\[
\osupp(ghg^*) \cup \osupp(g^*hg) \subseteq BB^{-1} \cup B^{-1}B \subseteq \Go.
\]
Therefore, $m \coloneqq g j(n) \in C_r^*(G)$ is a normaliser of $C_0(\Go)$ with $s(\gamma_1), \gamma_1^{-1} \gamma_2 \in S_m$. By condition~\cref{it: cond exp}, there exists a sequence $(f_k)_{k=1}^\infty \subseteq C_0(\Go)$ such that
\begin{equation} \label{eqn: E(m) limit}
\iota(E(m)) = \lim_{k \to \infty} (\iota(f_k) j(m)) = \lim_{k \to \infty} \iota(f_k E(m)).
\end{equation}
Since $s(\gamma_1) \in S_m \cap \Go = S_{E(m)}$, we have $E(m)(s(\gamma_1)) \ne 0$, and hence \cref{eqn: E(m) limit} implies that
\begin{equation} \label{eqn: f_k(s(gamma_1)) goes to 1}
1 = \frac{E(m)(s(\gamma_1))}{E(m)(s(\gamma_1))} = \lim_{k \to \infty} \frac{(f_k E(m))(s(\gamma_1))}{E(m)(s(\gamma_1))} = \lim_{k \to \infty} f_k(s(\gamma_1)) = \lim_{k \to \infty} \iota(f_k)(s(\gamma_1)).
\end{equation}
Since $\gamma_1^{-1}\gamma_2 \in S_m {\setminus} \Go = S_m {\setminus} S_{E(m)}$, we have
\begin{equation} \label{eqn: j(m) is not E(m)}
\iota(E(m))(\gamma_1^{-1} \gamma_2) = 0 \ne j(m)(\gamma_1^{-1} \gamma_2).
\end{equation}
However, \cref{eqn: E(m) limit,eqn: f_k(s(gamma_1)) goes to 1} together imply that
\begin{align*}
\iota(E(m))(\gamma_1^{-1} \gamma_2) &= \lim_{k \to \infty} (\iota(f_k) j(m))(\gamma_1^{-1} \gamma_2) \\
&= \lim_{k \to \infty} \iota(f_k)(s(\gamma_1)) j(m)(\gamma_1^{-1} \gamma_2) = j(m)(\gamma_1^{-1} \gamma_2),
\end{align*}
which contradicts \cref{eqn: j(m) is not E(m)}.

\Cref{it: untwisted LBH} $\implies$ \cref{it: effective}: Since this argument is more involved, we prove this as \cref{prop: untwisted LBH implies effective} in \cref{sec: Cstar-LBH implies effectiveness}.
\end{proof}

We conclude this section by showing that if the groupoid $G$ is ample, then for any twist $\Sigma$ over $G$, the conditional expectation $E\colon C_r^*(G;\Sigma) \to C_0(\Go)$ is \hl{implemented by projections}, in the sense that the functions $f_k \in C_0(\Go)$ appearing in condition~\cref{it: cond exp} of \cref{thm: main} can be chosen to be projections.

\begin{cor} \label{cor: ample}
Let $G$ be an ample locally compact Hausdorff groupoid. Then $G$ is effective if and only if for any twist $(\Sigma,i,q)$ over $G$ and any normaliser $n \in C_r^*(G;\Sigma)$ of $C_0(\Go)$, there exists a sequence of projections $(p_k)_{k=1}^\infty$ such that for each $k \ge 1$,
\[
\iota(p_k) j(n) = \iota(p_k E(n)) = \iota(E(n) p_k) = j(n) \iota(p_k) \in C_r^*(\Go;q^{-1}(\Go)),
\]
and
\[
\iota(E(n)) = \lim_{k \to \infty} (\iota(p_k) j(n)).
\]
\end{cor}

\begin{proof}
Suppose that $G$ is effective. Let $(\Sigma,i,q)$ be a twist over $G$, and fix a normaliser $n \in C_r^*(G;\Sigma)$ of $C_0(\Go)$. By the implication \cref{it: effective} $\implies$ \cref{it: twisted LBH} of \cref{thm: main}, we know that $S_n = q(\osupp(j(n)))$ is an open bisection of $G$. Following the proof of the implication \cref{it: twisted LBH} $\implies$ \cref{it: cond exp} in \cref{thm: main}, for each $k \in \N {\setminus} \{0\}$, let
\[
W_k \coloneqq \{ x \in \Go : \abs{E(n)(x)} > \tfrac{1}{k} \} = E(n)^{-1}\big(\C \setminus \overline{B}(0,\tfrac{1}{k})\big).
\]
Then for each $k \ge 1$, we have $W_k \subseteq \overline{W_k} \subseteq W_{k+1}$, and since $\overline{W_k}$ is compact and $G$ is ample, we can find a finite cover of $\overline{W_k}$ consisting of compact open subsets $U^{(k)}_1, \dotsc, U^{(k)}_{n_k}$ of $\Go$ contained in $W_{k+1}$. Then the union $B_k \coloneqq \bigcup_{\ell=1}^{n_k} U^{(k)}_\ell$ is a compact open subset of $\Go$ contained in $W_{k+1}$. For each $k \ge 1$, let $p_k \coloneqq 1_{B_k}$. Then each $p_k$ is a projection in $C_c(\Go,[0,1])$ such that $\supp(p_k) \subseteq S_{E(n)}$ and $p_k(x) = 1$ for all $x \in \overline{W_k}$. Thus the remainder of the proof follows exactly as in the proof of the implication \cref{it: twisted LBH} $\implies$ \cref{it: cond exp} in \cref{thm: main}.

Finally, since ample groupoids are \'etale, the converse follows trivially from the implication \cref{it: cond exp} $\implies$ \cref{it: effective} of \cref{thm: main}, which itself follows from \cref{prop: untwisted LBH implies effective}.
\end{proof}

\section{The algebraic vs the C*-algebraic local bisection hypothesis}
\label{sec: example}

In this section we present an example that will be used in the proof of \cref{prop: untwisted LBH implies effective} to show that if $G$ is not effective, then there is a normaliser $n \in C_r^*(G)$ of $C_0(\Go)$ such that $j(n)$ is not supported on a bisection. Our example also demonstrates how the algebraic and analytic situations differ. We think of condition~\cref{it: untwisted LBH} in \cref{thm: main} as the \hl{``untwisted'' C*-algebraic local bisection hypothesis} for $G$. Even if $G$ is ample, this is not the same as the ``local bisection hypothesis'' introduced by Steinberg in \cite[Definition~4.9]{Steinberg2019}: if $R$ is a commutative unital ring (for example, the complex numbers endowed with the discrete topology) and $A_R(G)$ is the Steinberg $R$-algebra associated to $G$, then we say that $G$ satisfies the \hl{(algebraic) local bisection hypothesis} if every normaliser $n \in A_R(G)$ of $A_R(\Go)$ is supported on an open bisection of $G$.

\begin{example} \label{eg: integers}
The integers satisfy the algebraic local bisection hypothesis but not the C*-algebraic local bisection hypothesis; that is, the integers do not satisfy the condition described in \cref{thm: main}\cref{it: untwisted LBH}.
\end{example}

\begin{proof}
The first claim follows from \cite[Corollary~9.3]{ACCCLMRSS2023}. To prove the second claim, we identify $C_r^*(\Z) = C^*(\Z)$ with $C(\T)$ via the Fourier transform $\FF\colon C^*(\Z) \to C(\T)$, which sends the generating unitary $u = \delta_1$ of $C^*(\Z)$ to the identity map on $\T$. The groupoid $G = \Z$ has unit space $\Go = \{0\}$, and so $\FF(C_0(\Go)) = \C 1_{C(\T)}$. Consider the function $m\colon \T \to \C$ given by
\[
m(z) = \frac{z - 2z^2}{\abs{z - 2z^2}}.
\]
Then $m$ is continuous and circle-valued since the zeros of $z-2z^2$ are $z = 0$ and $z = \frac{1}{2}$. Since $m^*m = mm^* = 1$, $m \in C(\T)$ is a normaliser of $\C 1_{C(\T)}$. However,
\[
j(\FF^{-1}(m)) = \frac{u - 2(u * u)}{\norm{u - 2(u * u)}} = \frac{\delta_1 - 2\delta_2}{\norm{\delta_1 - 2\delta_2}}
\]
is nonzero on both of the open bisections $\{1\}$ and $\{2\}$. So $S_{\FF^{-1}(m)} = \osupp(j(\FF^{-1}(m)))$ is not a bisection of $\Z$, and hence the C*-algebraic local bisection hypothesis does not hold.
\end{proof}

\section{The C*-algebraic local bisection hypothesis implies effectiveness}
\label{sec: Cstar-LBH implies effectiveness}

In this section, we complete the proof of \cref{thm: main} by establishing that \cref{it: untwisted LBH} implies \cref{it: effective}, restated in the following proposition.

\begin{prop} \label{prop: untwisted LBH implies effective}
Let $G$ be a locally compact Hausdorff \'etale groupoid. Suppose that for every normaliser $n \in C_r^*(G)$ of $C_0(\Go)$, $S_n = \osupp(j(n))$ is an open bisection of $G$. Then $G$ is effective.
\end{prop}

In our proof of \cref{prop: untwisted LBH implies effective}, we consider the possible orders of elements in $\Iso(G)^{\circ}$, the interior of the isotropy of the groupoid $G$. Most of the work comes in dealing with the existence of torsion. If every element of $\Iso(G)^{\circ}$ has infinite order, an argument like the one in \cref{eg: integers} makes the proof fairly straight forward. We begin with a technical lemma about normalisers in the group algebra of a group with prime order.

\begin{lemma} \label{lemma: prime group}
Let $p$ be a prime number, let $\Gamma \coloneqq \Z/p\Z$, let $\zeta$ be a nontrivial $p\th$ root of unity, and let $\delta_k\colon \Gamma \to \C$ be the point-mass function at $k \in \Gamma$.
\begin{enumerate}[label=(\alph*)]
\item \label{it: even prime} If $p = 2$, then $n \coloneqq \delta_0 - i \delta_1 \in C_r^*(\Gamma)$ is a normaliser of $\C \delta_0$ with $\osupp(n) = \Gamma$.
\item \label{it: odd prime} If $p \ne 2$, then $n \coloneqq \sum_{k = 0}^{p-1} \zeta^{k^2}\delta_k \in C_r^*(\Gamma)$ is a normaliser of $\C\delta_0$ with $\osupp(n) = \Gamma$.
\end{enumerate}
\end{lemma}

\begin{proof}
The proof of part~\cref{it: even prime} is straightforward: if $p = 2$, then
\[
n^*n = nn^* = (\delta_0 - i \delta_1)(\delta_0 + i \delta_1) = 2 \delta_0 \in \C \delta_0,
\]
and clearly $\osupp(n) = \Gamma$, so we are done. Now for part~\cref{it: odd prime}, let $p$ be an odd prime. Then for each $k \in \Gamma$, $n(k) = \zeta^{k^2} \ne 0$, and so $\osupp(n) = \Gamma$. We have
\[
n^* = \sum_{k=0}^{p-1} \zeta^{-k^2} \delta_{-k} = \sum_{k=0}^{p-1} \zeta^{-(p-k)^2} \delta_{p-k},
\]
and by letting $\ell = p - k$, we obtain
\[
n^* = \sum_{\ell=1}^p \zeta^{-\ell^2} \delta_\ell = \sum_{\ell=0}^{p-1} \zeta^{-\ell^2} \delta_\ell.
\]
Therefore,
\[
n n^* = \Big( \sum_{k=0}^{p-1} \zeta^{k^2} \delta_k \Big) \Big( \sum_{\ell=0}^{p-1} \zeta^{-\ell^2} \delta_\ell \Big) = \sum_{k=0}^{p-1} \Big( \sum_{\ell=0}^{p-1} \zeta^{\ell^2 - (k-\ell)^2} \Big) \delta_k = \sum_{k=0}^{p-1} \Big( \sum_{\ell=0}^{p-1} \zeta^{k(2\ell-k)} \Big) \delta_k.
\]
When $k = 0$, we have
\[
\sum_{\ell=0}^{p-1} \zeta^{k(2\ell-k)} = \sum_{\ell=0}^{p-1} 1 = p.
\]
We claim that when $k \ne 0$ is fixed, the terms in the sum
\begin{equation} \label{eqn: p sum terms}
\sum_{\ell=0}^{p-1} \zeta^{k(2\ell-k)}
\end{equation}
are a permutation of the $p\th$ roots of unity. To see this, note that if $k \in \{1, \dotsc, p-1\}$ and $k(2\ell_1-k) \equiv k(2\ell_2-k) \pmod{p}$ for some $\ell_1, \ell_2 \in \{0, \dotsc, p-1\}$, then we must have $\ell_1 \equiv \ell_2 \pmod{p}$, because $p$ is an odd prime. Therefore, each term in the sum~\labelcref{eqn: p sum terms} is a different power of $\zeta$. Since there are $p$ different terms in the sum, we must run through all of the powers of $\zeta$, which proves the claim. Now, since the sum of the $p\th$ roots of unity is $\frac{\zeta^p - 1}{\zeta - 1} = 0$, the coefficient of $\delta_k$ is $0$ for each $k \in \{1, \dotsc, p-1\}$. Thus $nn^* = p \delta_0 \in \C \delta_0$, and it follows that $n$ is a normaliser of $\C\delta_0$.
\end{proof}

The next lemma uses theory developed for groups that will help in our groupoid setting. We note that if a subset $S$ of a groupoid $G$ has the property that $S^N \subseteq \Go$ for some $N \in \N {\setminus} \{0\}$, and $r(\sigma), s(\sigma) \in S^N$ for every $\sigma \in S$, then we can view the collection of sets $\{S^k \colon k \in \Z \}$ as a cyclic group of order $N$ with identity $S^N$. Note that such a subset $S$ would necessarily be a bisection of $G$, because $SS^{-1} \cup S^{-1}S = S^N \subseteq \Go$.

\begin{lemma} \label{lemma: nice bisection}
Let $G$ be a locally compact Hausdorff \'etale groupoid. Suppose that $\gamma \in \Iso(G)^\circ$ satisfies $\gamma^N\in \Go$ for some $N \in \N {\setminus} \{0\}$ and that $N$ is the smallest such natural number. Then there exists an open bisection $B$ of $\Iso(G)^\circ$ with $\gamma \in B$ such that $B^N\subseteq \Go$. Thus $\{ B^k \colon k \in \Z \}$ is a cyclic group of order $N$ with identity $B^N$.
\end{lemma}

\begin{proof} Since $\gamma\in \Iso(G)^\circ$ and $G$ has a basis of open bisections, there exists an open bisection $L$ of $G$ such that $\gamma \in L \subseteq \Iso(G)^\circ$. Thus $W \coloneqq L^N \cap \Go$ is an open subset of $\Go$ containing $\gamma^N = s(\gamma)$. Take $B \coloneqq LW$. Then $B$ is an open bisection of $\Iso(G)^\circ$ containing $\gamma$. Since $L$ consists entirely of isotropy, we have $B^N = L^NW \subseteq \Go$, and $N$ is the smallest such natural number because $\gamma \in B$. For all $\sigma \in B \subseteq \Iso(G)^\circ$, we have $\sigma^N \in B^N \subseteq \Go$, and so $r(\sigma) = s(\sigma) = s(\sigma^N) = \sigma^N \in B^N$. It follows that $\{ B^k \colon k \in \Z \}$ is a cyclic group of order $N$ with identity $B^N$.
\end{proof}

We conclude this section by proving \cref{prop: untwisted LBH implies effective}, which completes the proof of \cref{thm: main}.

\begin{proof}[Proof of \cref{prop: untwisted LBH implies effective}]
We prove the contrapositive: if $G$ is not effective, then $G$ does not satisfy the C*-algebraic local bisection hypothesis. Suppose that $G$ is not effective. Then $\Iso(G)^\circ \ne \Go$. There are two cases to consider.

\emph{Case 1:} Suppose there exists $\gamma \in \Iso(G)^\circ {\setminus} \Go$ such that $\gamma^N\in \Go$ for some natural number $N > 1$, and that $N$ is the smallest such natural number. Define $B \subseteq \Iso(G)^\circ$ as in \cref{lemma: nice bisection}, and let $H$ be the cyclic group $\{B^k : k \in \Z\}$ of order $N$. Suppose that $p$ is a prime factor of $N$. Then there exists $c \in \N$ such that $\{ B^{c\ell} : \ell \in \Z \}$ is a cyclic subgroup of $H$ of order $p$.

Let $\Gamma \coloneqq \Z/p\Z$, and let $V \coloneqq B^c {\setminus} \Go$. Then $V$ is an open bisection of $G$, and $R \coloneqq \bigcup_{\ell=0}^{p-1} V^\ell \subseteq \Iso(G)^\circ$ is an open subgroupoid of $G$ with $R^{(0)} = V^0 = r(V)$. Thus, by \cite[Lemma~2.7]{BFPR2021}, there is an embedding $\psi\colon C_r^*(R) \hookrightarrow C_r^*(G)$.

We claim that if $k, \ell \in \{0, \dotsc, p-1\}$ and $k < \ell$, then $V^k \cap V^\ell = \varnothing$. To see this, suppose for contradiction that $\alpha, \beta \in V$ satisfy $\alpha^k = \beta^\ell$. Then $r(\alpha) = r(\beta)$, so $\alpha = \beta$ since $V$ is a bisection. Thus $\alpha^{\ell - k} = r(\alpha) \in \Go$. Also, since $V^p \subseteq \Go$, we have $\alpha^p \in \Go$. Let $d$ be the smallest positive natural number such that $\alpha^d \in \Go$. Then $0 < d \le \ell - k < p$. By the quotient-remainder theorem, there exist $q, t \in \N$ such that $0 \le t < d$ and $p = dq + t$. Thus $\alpha^t = \alpha^{p - dq} \in \Go$, which implies that $t = 0$, because $d > 0$ was chosen minimally. Thus $p = dq$, and since $p$ is prime and $d < p$, we must have $d = 1$. But then $\alpha = \alpha^d \in \Go$, which is a contradiction because $\alpha \in V = B^c {\setminus} \Go$. Thus $V^k \cap V^\ell = \varnothing$, as claimed. It follows that $R = \bigsqcup_{\ell=0}^{p-1} V^\ell \cong V^0 \rtimes \Gamma$ is a trivial $\Gamma$-bundle over $V^0$, and so by \cite[Lemma~2.73]{Williams2007}, we have $C_r^*(R) \cong C_0(V^0) \otimes C_r^*(\Gamma)$.

Fix $f \in C_c(V^0)$ with $f(r(\gamma)) = 1$. It is straightforward to check that for any normaliser $n \in C_r^*(\Gamma)$ of $\C \delta_0$, $f \otimes n \in C_r^*(R)$ is a normaliser of $C_0(V^0)$, and it follows that $\psi(f \otimes n) \in C_r^*(G)$ is a normaliser of $C_0(\Go)$. Now let $n \in C_r^*(\Gamma)$ be the normaliser of $\C \delta_0$ given in \cref{lemma: prime group}. Then $m \coloneqq \psi(f \otimes n) \in C_r^*(G)$ is a normaliser of $C_0(\Go)$ satisfying
\[
m(\alpha) = \begin{cases}
f(r(\alpha)) \, n(\ell) & \ \text{if } \alpha \in V^\ell \text{ for some } \ell \in \{0,\dotsc,p-1\} \\
0 & \ \text{else}.
\end{cases}
\]
However, $S_m = \osupp(j(m))$ is not a bisection of $G$, because $S_n = \osupp(n) = \Gamma$, and so $r(\gamma), \gamma \in S_m$. Thus the C*-algebraic local bisection hypothesis does not hold in this case.

\emph{Case 2:} Suppose that all elements of $\Iso(G)^\circ {\setminus} \Go$ have infinite order. Fix $\gamma \in \Iso(G)^\circ {\setminus} \Go$, and let $B$ be an open bisection of $G$ such that $\gamma \in B \subseteq \Iso(G)^\circ {\setminus} \Go$. Let $R \coloneqq \bigcup_{k \in \Z} B^k$. Then $R$ is an open subgroupoid of $G$ with $R^{(0)} = B^0 = r(B)$, and so by \cite[Lemma~2.7]{BFPR2021}, there is an embedding $\psi\colon C_r^*(R) \hookrightarrow C_r^*(G)$. Since all elements of $\Iso(G)^\circ {\setminus} \Go$ have infinite order, we have $B^\ell \cap B^k = \varnothing$ whenever $\ell \ne k$. It follows that $R = \bigsqcup_{k \in \Z} B^k = B^0 \rtimes \Z$ is a trivial $\Z$-bundle over $B^0$, and so by \cite[Lemma~2.73]{Williams2007}, we have $C_r^*(R) \cong C_0(B^0) \otimes C_r^*(\Z)$.

Now apply the argument from \cref{eg: integers} to obtain a normaliser $n \in C_r^*(\Z)$ of $\C \delta_0$ such that $1, 2 \in S_n = \osupp(j(n))$; that is $S_n$ is not a bisection of $\Z$. Similar to Case 1, fix $f \in C_c(B^0)$ with $f(r(\gamma)) = 1$. Then $f \otimes n \in C_r^*(R)$ is a normaliser of $C_0(B^0)$, and it follows that $m \coloneqq \psi(f \otimes n) \in C_r^*(G)$ is a normaliser of $C_0(\Go)$ satisfying
\[
m(\alpha) = \begin{cases}
f(r(\alpha)) \, n(k) & \ \text{if } \alpha \in V^k \text{ for some } k \in \Z \\
0 & \ \text{else}.
\end{cases}
\]
However, $S_m = \osupp(j(m))$ is not a bisection of $G$, because both $1, 2 \in S_n$, and so $\gamma, \gamma^2 \in S_m$. Thus the C*-algebraic local bisection hypothesis does not hold in this case either.
\end{proof}

\vspace{2ex}
\bibliographystyle{amsplain}
\makeatletter\renewcommand\@biblabel[1]{[#1]}\makeatother
\bibliography{ABCCLMR_references}

\providecommand{\bysame}{\leavevmode\hbox to3em{\hrulefill}\thinspace}
\providecommand{\MR}{\relax\ifhmode\unskip\space\fi MR }
\providecommand{\MRhref}[2]{%
  \href{http://www.ams.org/mathscinet-getitem?mr=#1}{#2}
}
\providecommand{\href}[2]{#2}
\begin{thebibliography}{10}

\bibitem{Armstrong2022}
B.~Armstrong, \emph{A uniqueness theorem for twisted groupoid {C*-algebras}},
  J. Funct. Anal. \textbf{283} (2022), 1--33, \doi{10.1016/j.jfa.2022.109551}.

\bibitem{ACCCLMRSS2023}
B.~Armstrong, G.G. {\noopsort{Castro}{de Castro}}, L.O. Clark, K.~Courtney,
  Y.-F. Lin, K.~McCormick, J.~Ramagge, A.~Sims, and B.~Steinberg,
  \emph{{Reconstruction of twisted Steinberg algebras}}, Int. Math. Res. Not.
  IMRN \textbf{2023} (2023), 2474--2542, \doi{10.1093/imrn/rnab291}.

\bibitem{ACCLMR2022}
B.~Armstrong, L.O. Clark, K.~Courtney, Y.-F. Lin, K.~McCormick, and J.~Ramagge,
  \emph{Twisted {Steinberg} algebras}, J. Pure Appl. Algebra \textbf{226}
  (2022), 1--33, \doi{10.1016/j.jpaa.2021.106853}.

\bibitem{BCFS2014}
J.H. Brown, L.O. Clark, C.~Farthing, and A.~Sims, \emph{Simplicity of algebras
  associated to \'etale groupoids}, Semigroup Forum \textbf{88} (2014),
  433--452, \doi{10.1007/s00233-013-9546-z}.

\bibitem{BFPR2021}
J.H. Brown, A.H. Fuller, D.R. Pitts, and S.A. Reznikoff, \emph{{Graded
  C*-algebras and twisted groupoid C*-algebras}}, New York J. Math. \textbf{27}
  (2021), 205--252.

\bibitem{CFST2014}
L.O. Clark, C.~Farthing, A.~Sims, and M.~Tomforde, \emph{A groupoid
  generalisation of {Leavitt} path algebras}, Semigroup Forum \textbf{89}
  (2014), 501--517, \doi{10.1007/s00233-014-9594-z}.

\bibitem{CaH2012}
L.O. Clark and A.~{\noopsort{Huef}{an Huef}}, \emph{The representation theory
  of {C*-algebras} associated to groupoids}, Math. Proc. Cambridge Philos. Soc.
  \textbf{153} (2012), 167--191, \doi{10.1017/S0305004112000047}.

\bibitem{aHKS2011}
A.~{\noopsort{Huef}{an Huef}}, A.~Kumjian, and A.~Sims, \emph{{A
  Dixmier--Douady theorem for Fell algebras}}, J. Funct. Anal. \textbf{260}
  (2011), 1543--1581, \doi{10.1016/j.jfa.2010.11.011}.

\bibitem{Kumjian1986}
A.~Kumjian, \emph{On {C*-diagonals}}, Canad. J. Math. \textbf{38} (1986),
  969--1008, \doi{10.4153/cjm-1986-048-0}.

\bibitem{KM2020}
B.K. Kwa\'{s}niewski and R.~Meyer, \emph{{Noncommutative Cartan
  C*-subalgebras}}, Trans. Amer. Math. Soc. \textbf{373} (2020), 8697--8724,
  \doi{10.1090/tran/8174}.

\bibitem{MW1992}
P.S. Muhly and D.P. Williams, \emph{{Continuous trace groupoid C*-algebras,
  II}}, Math. Scand. \textbf{70} (1992), 127--145,
  \doi{10.7146/math.scand.a-12390}.

\bibitem{Pitts2022}
D.R. Pitts, \emph{Normalizers and approximate units for inclusions of
  {C*-algebras}}, to appear in Indiana Univ. Math. J.,
  \doi{10.48550/arXiv.2109.00856}.

\bibitem{Raad2022}
A.I. Raad, \emph{A generalization of {R}enault's theorem for {C}artan
  subalgebras}, Proc. Amer. Math. Soc. \textbf{150} (2022), 4801--4809,
  \doi{10.1090/proc/16003}.

\bibitem{Renault1980}
J.~Renault, \emph{{A Groupoid Approach to C*-Algebras}}, Lecture Notes in
  Math., vol. 793, Springer, Berlin, 1980, \doi{10.1007/bfb0091072}.

\bibitem{Renault2008}
\bysame, \emph{Cartan subalgebras in {C*-algebras}}, Irish Math. Soc. Bull.
  \textbf{61} (2008), 29--63, \doi{10.33232/bims.0061.29.63}.

\bibitem{Sims2020}
A.~Sims, \emph{Hausdorff \'etale groupoids and their {C*-algebras}}, Operator
  algebras and dynamics: groupoids, crossed products, and Rokhlin dimension
  (F.~Perera, ed.), Advanced Courses in Mathematics, CRM Barcelona,
  Birkh\"auser, 2020, \doi{10.1007/978-3-030-39713-5}.

\bibitem{Steinberg2010}
B.~Steinberg, \emph{A groupoid approach to discrete inverse semigroup
  algebras}, Adv. Math. \textbf{223} (2010), 689--727,
  \doi{10.1016/j.aim.2009.09.001}.

\bibitem{Steinberg2019}
\bysame, \emph{Diagonal-preserving isomorphisms of \'etale groupoid algebras},
  J. Algebra \textbf{518} (2019), 412--439,
  \doi{10.1016/j.jalgebra.2018.10.024}.

\bibitem{Williams2007}
D.P. Williams, \emph{{Crossed Products of C*-Algebras}}, Math. Surveys and
  Monographs, vol. 134, Amer. Math. Soc., Providence, RI, 2007,
  \doi{10.1090/surv/134}.

\end{thebibliography}
\vspace{-3ex}

\end{document}